\documentclass[12pt,leqno]{amsart}
\newtheorem{thm}{Theorem}[section]
\newtheorem{defi}{Definition}[section]

\newtheorem{cor}{Corollary}[section]
\newtheorem{pr}{Proposition}[section]
\newtheorem{rem}{Remark}[section]

\newcommand{\be}{\begin{equation}}
\newcommand{\ee}{\end{equation}}
\newcommand{\bea}{\begin{eqnarray}}
\newcommand{\eea}{\end{eqnarray}}
\newcommand{\beb}{\begin{eqnarray*}}
\newcommand{\eeb}{\end{eqnarray*}}
\usepackage{amssymb,amsfonts,amsthm,setspace,indentfirst,mathrsfs}
\usepackage{minibox}
\usepackage{enumitem}
\numberwithin{equation}{section}
\begin{document}
%%%%%%%%%%%%%%%%%%%%%%%%%%%%%%%%%%%%%%%%%%%%%%%%%%%%%%%%%%%%%%%%%%%%%%%%%%%%%%%%%%%%%%%%%%%%%%%%%%%%%%%%%%%%%
%
\title[On Curvature properties of Nariai Spacetimes]{On Curvature properties of Nariai Spacetimes}
\author[A. A. Shaikh, Akram Ali, Ali H. Alkhaldi and D. Chakraborty]{Absos Ali Shaikh$^1$, Akram Ali$^2$, Ali H. Alkhaldi$^2$ and Dhyanesh Chakraborty$^1$}
\date{}
\address{\noindent\newline$^1$ Department of Mathematics,
\newline University of Burdwan, 
\newline Golapbag, Burdwan-713104,
\newline West Bengal, India} 
\email{aask2003@yahoo.co.in, aashaikh@mathburuniv.ac.in}
\email{dhyanesh2011@gmail.com}
\address{\noindent\newline$^2$ Department of Mathematics,
	\newline College of Science,
	\newline King Khalid University, 
	\newline 9004 Abha, Saudi Arabia} 
\email{akramali133@gmail.com ; akali@kku.edu.sa}
\email{ahalkhaldi@kku.edu.sa}
\dedicatory{}
%
%%%%%%%%%%%%%%%%%%%%%%%%%%%%%%%%%%%%%%%%%%%%% Abstract %%%%%%%%%%%%%%%%%%%%%%%%%%%%%%%%%%%%%%%%%%%%%%%%%%%%%
\begin{abstract}
	This article is concerned with the study of the geometry of (charged) Nariai spacetime, a topological product spacetime, by means of covariant derivative(s) of its various curvature tensors. It is found that on this spacetime the condition $\nabla R=0$ is satisfied  and it also admits the pseudosymmetric type curvature conditions $C\cdot R=\frac{(1+L_0)}{6r_0^2}Q(g,R)$ and $P\cdot R=-\frac{1}{3}Q(S,R)$. Moreover, it is $4$-dimensional Roter type, $2$-quasi-Einstein and generalized quasi-Einstein manifold. The energy-momentum tensor is expressed explicitly by some $1$-forms. It is worthy to see that a generalization of such topological product spacetime proposes to exist a class of generalized recurrent type manifolds which is semisymmetric.
\end{abstract}
%%%%%%%%%%%%%%%%%%%%%%%%%%%%%%%%%%%%%%%%%%%%%%%%%%%%%%%%%%%%%%%%%%%%%%%%%%%%%%%%%%%%%%%%%%%%%%%%%%%%%%%%%%%%
%
\maketitle
%

%%%%%%%%%%%%%%%%%%%%%%%%%%%%%%%%%%%%%%%%%%%%%%%%%%%%%%%%%%%%%%%%%%%%%%%%%%%%%%%%%%%%%%%%%%%%%%%%%%%%%%%%%%%%%%
%																					Introduction
%%%%%%%%%%%%%%%%%%%%%%%%%%%%%%%%%%%%%%%%%%%%%%%%%%%%%%%%%%%%%%%%%%%%%%%%%%%%%%%%%%%%%%%%%%%%%%%%%%%%%%%%%%%%%%
\section{\bf Introduction}\label{intro}
%%%%%%%%%%%%%%%%%%%%%%%%%%%%%%%%%%%%%%%
\indent In the literature of general relativity, an important type of analytical solution of Einstein field equations with positive cosmological constant was introduced by Nariai \cite{N50, N51} in $1951$. Mathematcally, this spacetime can be described as the topological product $dS_2\times S^2$, where $dS_2$ is a $2$-dimensional de-Sitter spacetime and $ S^2$ is a round $2$-sphere of constant radius. Such spacetime became a crucial part of study in both theoretical and mathematical physics when Ginsparg and Perry \cite{GP83} linked it to the dS-Schwarzschild black hole solution during the study of thermodynamical equillibrium of the black hole. They described that the Nariai solution is generated if the black hole event horizon of the dS-Schwarzschild solution approaches to the cosmological horizon through an appropriate limiting procedure. In fact, the Nariai solution is geodesically complete and it has $D$-type  Weyl tensor in Petrov classification. In terms of spherical coordinates $(t,r, \theta, \phi)$, the metric of the Nariai spacetime is given by
\bea\label{NNM}
ds^2=R_0^2(-sin^2rdt^2+dr^2+d\theta^2+sin^2\theta d\phi^2),
\eea 
where $R_0^2=\frac{1}{\Lambda}$, $\Lambda$ is a positive cosmological constant. This spacetime has been extended to the charged Nariai spacetime by Bertotti-Robinson \cite{B59, R59} to include Maxwell field. According to Hawking and Ross \cite{HR95}, the charged Nariai spacetime can also be obtained from the dS-Reissner-Nordstr{\"o}m black hole by taking appropriate limit of cosmological horizon going into the outer charged black hole horizon. The charged Nariai metric with respect to spherical symmetry is given by
\bea\label{CNM}
ds^2=\frac{r_0^2}{L_0}(-sin^2rdt^2+dr^2)+r_0^2(d\theta^2+sin^2\theta d\phi^2),
\eea
where $r_0 > 0$ and $0 < L_0\le 1$ are two constants. The Maxwell fields of the solution are 

\bea\label{MF}
\mathcal{F}&=&\frac{q_0}{L_0} sin \,r\,dt\wedge dr \;\; \; \; \mbox{for the electric case,}\\
\mathcal{F}&=& q_0 sin \,\theta \,d\theta\wedge d\phi \; \; \; \; \mbox{for the magnetic case}
\eea
$q_0$ being the magnetic or electric charge respectively. The cosmological constant and the charge of the spacetime are related to the constants $r_0$ and $L_0$ by
\bea\label{re}
q_0^2&=&\frac{1-L_0}{2}r_0, \\
\Lambda&=&\frac{1+L_0}{2r_0^2}.
\eea
We briefly denote the Nariai spacetime as $NS$ and the charged Nariai spacetime as $CNS$ respectively. From (\ref{CNM})--(\ref{re}) it is obvious that for $L_0=1$ the $CNS$ metric reduces to the $NS$ metric (\ref{NNM}). Also the limit $\Lambda=0$ takes the $NS$ metric (\ref{NNM}) into the Minkoswki spacetime metric.\\
\indent Many authors have studied the $NS$ and $CNS$ solutions for cosmological observation, interpretation and generalization. Bousso \cite{B97} extended the Nariai solution to the dilation charged Nariai solution. Florian Beyer \cite{B09I, B09II} studied the asymtotics of the generalized Nariai solution and made valuable suggestion that the Nariai solutions are non-generic among general solutions of Einstein field equations in vacuum with positive cosmological constant. The propagation of non-expanding impulsive waves in the Nariai universe were studied by Ortaggio \cite{O02}. However, the geometric structures of such a spacetime is yet to known which is the moto of this paper.\\
\indent It is well known that in general relativity the gravity of a spacetime is fully determined by its energy-momentum tensor and this tensor is related to the curvature tensors of that spacetime. Hence, to describe the gravity of a space we need to know its curvature tensors and their properties and the theory of differential geometry provides such curvature properties along with several generalizations. The notion of symmetry, defined locally as $\nabla R=0$ by Cartan \cite{C26, C27} has a great importance as it describes the isometry of all local geodesics at a given point of a manifold. The notion of semisymmetry ($R\cdot R=0$) was introduced by Cartan \cite{C46}, which generalizes the concept of local symmetry. We review such geometric structures along with their various generalizations in section $2$. However, in section $3$ we determine the geometric structures of $CNS$ metric \eqref{CNM} and $NS$ metric \eqref{NNM}. It is shown that $CNS$ is locally symmetric, pseudosymmetric due to conformal tensor and it satisfies $P\cdot R=-\frac{1}{3}Q(S,R)$. In section $3$ a generalization of $CNS$ metric is considered and it is interesting to see that such metric is semisymmetric and super generalized recurrent manifold with recurrent conformal curvature tensor. Finally we leave a conclusion in section $4$.
%%%%%%%%%%%%%%%%%%%%%%%%%%%%%%%%%%%%%%%%%%%%%%%%%%%%%%%%%%%%%%%%%%%%%%%%%%%%%%%%%%%%%%%%%%%%%%%%%%%%%%%%%%%%%%%%%%%%
%                                                Geometric structures
%%%%%%%%%%%%%%%%%%%%%%%%%%%%%%%%%%%%%%%%%%%%%%%%%%%%%%%%%%%%%%%%%%%%%%%%%%%%%%%%%%%%%%%%%%%%%%%%%%%%%%%%%%%%%%%%%%%%
\section{\bf{Elementaries}}
%==========================
 In this section we discuss various curvature tensors and some of their related geometric structures which are essential to study the curvature restricted geometric structures of the charged Nariai spacetime. We consider an  $n$-dimensional ($n\geq 3$) connected semi-Riemannian smooth manifold $M$ equipped with the metric $g$. Suppose $\nabla$ and $R$ are respectively the operator of covariant differentiation and Reimann curvature tensor of $M$. Now we define the $(1,3)$-tensors $C^s_{pqr}$, $P^s_{pqr}$, $W^s_{pqr}$ and $K^s_{pqr}$ of $M$ respectively as \cite{DDVV91, DG02, DHJKS14, SK19}
 \bea
 C^s_{pqr} &=& R^s_{pqr} + \frac{1}{(n-2)}(\delta^s_{p}S_{qr} - \delta^s_{q}S_{pr} + S^s_{p}g_{qr} - S^s_{q}g_{pr}) - \frac{\kappa}{(n-1)(n-2)}(\delta^s_{p}g_{qr} - \delta^s_{q}g_{pr})\nonumber,\\
 P^s_{pqr} &=& R^s_{pqr} - \frac{1}{(n-1)}(\delta^s_{p}S_{qr} - \delta^s_{q}S_{pr}),\nonumber\\
 W^s_{pqr} &=& R^s_{pqr} - \frac{\kappa}{n(n-1)}(\delta^s_{p}g_{qr} - \delta^s_{q}g_{pr})\; \; \mbox{and} \nonumber\\
 K^s_{pqr} &=& R^s_{pqr} - \frac{1}{(n-2)}(\delta^s_{p}S_{qr} - \delta^s_{q}S_{pr} + S^s_{p}g_{qr} - S^s_{q}g_{pr})\nonumber
 \eea
 where $S_{pq}$ is the Ricci tensor and $S^s_p$ is the Ricci operator defined as $S_{pq}=g_{ps}S^s_{q}$. Again the Ricci tensors of level $k$ are defined as $S^k_{pq}=S_{ps}^{k-1}S^s_{q}$.\\ The $(0,4)$-tensor $D_{pqrs}$ corresponding to a $(1,3)$-tensor $D^s_{pqr}$ is defined as 
 \bea\label{tns}
 D_{pqrs}=g_{st}D^t_{pqr}.
 \eea 
 Replacing $D^t_{pqr}$ in \eqref{tns} by $C^t_{pqr}$ (resp., $P^t_{pqr}$, $W^t_{pqr}$ and $K^t_{pqr})$ we obtain the $(0,4)$ conformal (resp., projective, concircular and conharmonic) curvature tensor. Let $U$ be a $(0,4)$-tensor and $Z$ be a symmetric $(0,2)$-tensor on $M$. Then for a $(0,k)$, $k\ge 1$, tensor $T$ we define the $(0,k+2)$-tensors $U\cdot T$ and $Q(Z,T)$ \cite{DG02, DGHS98, SDHJK15, SK14, T74} respectively as
 \bea
 (U\cdot T)_{p_1p_2...p_krs} &=& -g^{\alpha \beta}[U_{rsp_1\beta}T_{\alpha p_2...p_k} +\cdots +U_{rsp_k\beta}T_{p_1p_2...\alpha}],\nonumber \\
 Q(Z,T)_{p_1p_2...p_k rs} &=& Z_{s p_1}T_{r p_2...p_k} + \cdots + Z_{s p_k}T_{p_1p_2...r} \nonumber \\ &-& Z_{r p_1}T_{s p_2...p_k} - \cdots - Z_{r p_k}T_{p_1p_2...s}\nonumber.
 \eea
 \begin{defi}\label{defi2.1}(\cite{AD83, ADEHM14, D92, DDVV94, DHV04, SAA18, SBKpp, SK14, SK17, SK18, SKppsnw,S82, S84, S85})
 	The manifold $M$ is said to be $T$-pseudosymmetric type manifold due to $U$ if $U\cdot T=\mathcal J_T Q(Z,T)$, where  $\mathcal J_T$ is some scalar function on $\left\{x\in M:Q(Z,T)_x\ne 0\right\}$ and is said to be $T$-semisymmetric manifold due to $U$ if $U\cdot T=0$ holds on $M$.
 \end{defi}
 In particular, for $U=R$, $Z=g$ and $T=R$ a $T$-pseudosymmetric manifold is called Dezcz pseudosymmetric manifold and for $Z=S$ it is called Ricci generalized pseudosymmetric. For other curvature tensors we obtain the corresponding pseudosymmetric and semisymmetric type curvature conditions.
A semi-Reimannian manifold $M$ on which the relation $S=\frac{\kappa}{n} g$ holds is called Einstein manifol and the generalization of such notion is given in the following definition:
\begin{defi}\label{defi2.2}(\cite{S09, SHY09, SKH11}, \cite{SK17}--\cite{SK19})
If rank of $(S_{pq}-\alpha g_{pq})$ is $m$, $0\leq m\leq (n-1)$, for a scalar $\alpha$, then a semi-Reimannian manifold is called $m$-quasi Einstein manifold and in particular for $m=1$ it is quasi-Einstein manifold.
\end{defi}
 
We give another generalization of the notion of Einstein manifolds as follows:
\begin{defi}\label{2.5}(\cite{B87, SK19})
The manifold $M$ is said to be Ein(4) if
\bea
\alpha_0S^4_{pq}+\alpha_1S^3_{pq}+\alpha_2S^2_{pq}+\alpha_3S+\alpha_4g_{pq} = 0, \; \; \alpha_0 \ne 0\nonumber
\eea
holds on $M$ for some scalars $\alpha_i$, $0\leq i\leq 4$. We obtain Ein(3) and Ein(2) manifolds for $\alpha_0 =0$ and $\alpha_0=\alpha_1=0$ respectively. 
\end{defi}
We note that G$\ddot{\mbox{o}}$del spacetime \cite{DHJKS14}, Siklos spacetime \cite{SDD} are quasi-Einstein and $Ein(2)$ manifolds whereas Som-Roychoudhury spacetime \cite{SK16108}, Lifshitz spacetime \cite{SSC19} are $2$-quasi Einstein and $Ein(3)$ manifolds. For curvature properties of Vaidya metric, we refer the reader to see \cite{SKS17}.
\begin{defi}\label{2.6}(\cite{DHJKS14, G78, SB, S81})
If the Ricci tensor of a semi-Riemannian manifold $M$ admits the relation
\bea
\sum_{p,q,r} \nabla_pS_{qr}=0 \nonumber\\
(resp., \nabla_pS_{qr}-\nabla_qS_{pr} =0)\nonumber
 \eea
on $M$, $\sum$ being the cyclic sum over $p$, $q$, $r$, then it is said to be Cyclic parallel (resp., Codazzi type) Ricci tensor.
\end{defi} 
\indent The tensor $E\wedge A$, the Kulkarni-Nomizu product, of two symmetric $(0,2)$ tensors $E$ and $A$ is defined as \cite{DG02, DGHS11, DGHS98, DH03, G02, SRK16} 
\bea
(E\wedge A)_{pqrs} = E_{ps}A_{qr} - E_{pr}A_{qs} + E_{qr}A_{ps} - E_{qs}A_{pr}.\nonumber
\eea
\begin{defi}\label{defi2.4}
The manifold $M$ is said to be a generalized Roter type manifold (\cite{DGJPZ13, DGJZ16, DGPV15, DGPV11, S15, SDHJK15, SK16, SK19}) if $$R_{pqrs}=\mu_1(g\wedge g)_{pqrs}+\mu_2(g\wedge S)_{pqrs}+\mu_3(S\wedge S)_{pqrs}+\mu_4(g\wedge S^2)_{pqrs}+\mu_5(S\wedge\ S^2)_{pqrs}+\mu_6(S^2\wedge S^2)_{pqrs}$$ holds for some scalars $\mu_i$, $1\leq i\leq 6$. It is Roter type manifold for $\mu_4=\mu_5=\mu_6=0$ (\cite{D03, D0374, DGHS11, DGPV11, DPS13, DVV91, G07}).
\end{defi}
\begin{defi}
The manifold $M$ is said to be weakly symmetric \cite{TB89} (see, \cite{MS16, SDHJK15, SK12} and also references therein) if
\beb
\nabla_\alpha  R_{pqrs} = \Pi_{\alpha} R_{pqrs} + \Phi_{p} R_{\alpha qrs} + \overline \Phi_{q} R_{p\alpha rs} + \Psi_{r} R_{pq\alpha s} + \overline \Psi_{s} R_{pqr\alpha}
\eeb
holds for some associated covectors $\Pi_{\alpha}, \Phi_{\alpha}, \overline \Phi_{\alpha}, \Psi_{\alpha}$ and $\overline \Psi_{\alpha}$. If $\frac{1}{2}\Pi_{\alpha} = \Phi_{\alpha} = \overline \Phi_{\alpha} = \Psi_{\alpha} = \overline \Psi_{\alpha}$ then it is called Chaki pseudosymmetric manifold \cite{C87}.
\end{defi}
\indent It is also noted that the notion of Chaki pseudosymmetry is different from Deszcz pseudosymmetry.
\begin{defi}(\cite{SAR13, SKA16, SP10, SR10, SR11, SRK15})
Let $T$ be a $(0,4)$-tensor of $M$. Then $M$ is called a $T$-super generalized recurrent manifold if
 \bea\label{sgr}
\nabla_{\alpha} T_{pqrs}=\Pi_{\alpha} \otimes T_{pqrs}+ \Phi_{\alpha} \otimes (S\wedge S)_{pqrs}+ \Psi_{\alpha} \otimes (g\wedge S)_{pqrs} + \Theta_{\alpha} \otimes (g\wedge g)_{pqrs}
\eea
holds on M for some associated covectors $\Pi_{\alpha}$, $\Phi_{\alpha}$, $\Psi_{\alpha}$, $\Theta_{\alpha}$. In particular for $\Phi_{\alpha}=0$ (resp., $\Psi_{\alpha}=0$,and $\Phi_{\alpha}=\Psi_{\alpha}=\Theta_{\alpha}=0$) it is called $T$-hyper generalized recurrent manifold (resp., $T$-weakly generalized recurrent and $T$-recurrent manifold \cite{R46, R49, R4950, W50}).
\end{defi}
A $T$-recurrent (resp., $T$-weakly generalized recurrent, $T$-hyper generalized recurrent and $T$-super generalized recurrent) manifold is briefly denoted as $T$-$K_n$ (resp., $T$-$WGR_n$, $T$-$HGK_n$ and $T$-$SGK_n$). In particular for $T=R$, $T$-$K_n$ (resp., $T$-$WGR_n$, $T$-$HGK_n$ and $T$-$SGK_n$) is simply denoted as $K_n$ (resp., $WGR_n$, $HGK_n$ and $SGK_n$) and is called recurrent (resp., weakly generalized recurrent, hyper generalized recurrent and super generalized recurrent ) manifold.
\begin{defi}\label{cmp}(\cite{MM12, MM12128})
Let $B$ be a $(0,4)$-tensor of $M$. Then the Ricci tensor of $M$ is said to be $B$-compatible if 
\bea
\sum_{p,q,r}S^t_{p}B_{qrst}=0
\eea
$\sum$ being the cyclic sum over $p$, $q$, $r$, holds on $M$. Again an $1$-form $A$ is said to be $B$-compatible if $A\otimes A$ is $B$-compatible.
\end{defi}
\indent Substituting $B$ by $R$, $C$, $P$, $W$ and $K$ we can obtain Riemann compatibility, conformal compatibility, concircular compatibility and conharmonic compatibility respectively. Also the curvature $2$-forms for the tensor $B$ are recurrent \cite{B87, LR89, MS12, MS13, MS14} if
\bea
\sum_{p,q,r} \nabla_p B_{qrs\alpha} = \sum_{p,q,r}\Pi_{p} B_{qrs\alpha}\nonumber
\eea
holds and for a $(0,2)$-tensor $A$, the corresponding $1$-foms are recurrent if $\nabla_p A_{qr} - \nabla_q A_{pr} = \Pi_p A_{qr} - \Pi_q A_{pr}$.
\begin{defi}\label{defi2.9}(\cite{P95, SK19, V85})
Let the relation 
\bea
\sum_{p,q,r} \Phi_p B_{qrst} =0,\nonumber
\eea
$\sum$ being the cyclic sum over $p$, $q$, $r$, holds on $M$ and $\mathcal L(M)$ be the vector space formed by all such covectors. Then $M$ is called $B$-space by Venzi if $dim\;\mathcal{L}(M)\geq1$.
\end{defi} 
%
%%%%%%%%%%%%%%%%%%%%%%%%%%%%%%%%%%%%%%%%%%%%%%%%%%%%%%%%%%%%%%%%%%%%%%%%%%%%%%%%%%%%%%%%%%%%%%%%%%%%%%%%%%%%%%%%%%%%%%%%%%%%
%                                          Curvature Restricted Geometric Structure                                        %
%%%%%%%%%%%%%%%%%%%%%%%%%%%%%%%%%%%%%%%%%%%%%%%%%%%%%%%%%%%%%%%%%%%%%%%%%%%%%%%%%%%%%%%%%%%%%%%%%%%%%%%%%%%%%%%%%%%%%%%%%%%%
%
\section{\bf{CNS admitting geometric properties }}
%=======================================================
%
The non-zero components of the metric tensor of $CNS$ metric \eqref{CNM} are
\be\label{g}
g_{11} =-\frac{r_0^2}{L_0}sin^2 r, \ \ g_{22}=\frac{r_0^2}{L_0}, \ \ g_{33}=r_0^2, \ \ g_{44}= r_0^2 sin^2\theta.
\ee
In view of Section $2$ we can calculate the non-vanishing local expressions (upto symmetry) of $R_{\alpha \beta \gamma \delta}$, $S_{\alpha \beta}$ and $\kappa$ as
%==============    R  , S, \kappa   ==================%
\be\label{rs}
\left\{\begin{array}{c}
R_{1212}=\frac{r_0^2}{L_0}sin^2 r, \; R_{3434}=r_0^2 sin^2\theta ;\\ S_{11}=sin^2r,\; S_{22}=S_{33}=-1, \; S_{44}=-sin^2\theta; \\ \kappa=-\frac{2(1+L_0)}{r_0^2}.
\end{array}\right.
\ee
Also from \eqref{g} and \eqref{rs}, one can easily obtain the tensors $(g\wedge g)_{\alpha \beta \gamma \delta}$, $(g\wedge S)_{\alpha \beta \gamma \delta}$, $(S\wedge S)_{\alpha \beta \gamma \delta}$ as
\be
\begin{array}{c}
	(g\wedge g)_{1212}=\frac{2r_0^4sin^2r}{L_0^2}=L_0(g\wedge g)_{1313}, \; \; (g\wedge g)_{1414}=\frac{2r_0^4sin^2rsin^2r}{L_0},\;\; (g\wedge g)_{2323}=-\frac{2r_0^4}{L_0},\\ (g\wedge g)_{2424}=-\frac{2r_0^4sin^2\theta}{L_0}=L_0(g\wedge g)_{3434};\;\; (g\wedge S)_{1212}=-\frac{2r_0^2sin^2r}{L_0}=\frac{2}{(1+L_0)}(g\wedge S)_{1313},\\ (g\wedge S)_{1414}=-\frac{(1+L_0)}{L_0}r_0^2sin^2rsin^2\theta, \;\; (g\wedge S)_{2323}=\frac{(1+L_0)}{L_0}r_0^2, \\ (g\wedge S)_{2424}=\frac{(1+L_0)}{L_0}r_0^2sin^2\theta=\frac{(1+L_0)}{2L_0}(g\wedge S)_{3434};\;\; (S\wedge S)_{1212}=2sin^2r=(S\wedge S)_{1313}, \\ (S\wedge S)_{1414}=2sin^2rsin^2\theta, \;\; (S\wedge S)_{2323}=-2, \;\; (S\wedge S)_{2424}=-2sin^2\theta =(S\wedge S)_{3434}. 
\end{array}\nonumber
	\ee
These tensors decompose the tensor $R$ as
\be\label{RT}
R= \left( -\frac{(1+L_0)L_0^2}{2r_0(L_0-1)^2}\right) \; (g\wedge g)  + \left( -\frac{2L_0}{(L_0-1)^2}\right) \; (g\wedge S) + \left( -\frac{(1+L_0)r_0^2}{2(L_0-1)^2}\right). \; (S\wedge S)
\ee
Hence, the $CNS$ metric \eqref{CNM} is Roter type. Also it is easy to see that $\nabla S=0$ and from \eqref{RT} we find $\nabla R=0$.
\begin{rem}
	Since CNS is Roter type satisfying the relation \eqref{RT}, from $Theorem$ $6.7$ of \cite{DGHS11} we get the following properties:\\
	(i)  $\mu_1 S^2+ \mu_2 S+ \mu_3 g =0$ \;for\; $\mu_1=1$,\; $\mu_2=\frac{(1+L_0)}{r_0^2}$, \; $\mu_3=\frac{L_0}{r_0^4}$,\\
	(ii) $C\cdot R = \mathcal J_R Q(g,R)$ \; for \; $\mathcal J_R =\frac{(1+L_0)}{6L_0^2}$ \; and \\
	(iii) $Q(S,R)=\frac{2L_0}{(1+L_0)r_0^2}\; Q(g,R)$.
	\end{rem}
Again from \eqref{g} and \eqref{rs} we can calculate:
%
%==============    C, P   =====================%
\be\label{C}
\left\{\begin{array}{c}
	C_{1212}=-\frac{(1+L_0)r_0^2sin^2 r}{3L_0^2}, \; C_{1313}=\frac{(1+L_0)r_0^2sin^2 r}{6L_0}, \; C_{1414}=\frac{(1+L_0)r_0^2sin^2 r sin^2\theta}{6L_0}, \\ C_{2323}=-\frac{(1+L_0)r_0^2}{6L_0}, \; C_{2424}=-\frac{(1+L_0)r_0^2sin^2\theta}{6L_0}, \; C_{3434}=\frac{(1+L_0)r_0^2sin^2\theta}{3}.
\end{array}\right.
\ee
\be\label{P}
\left\{\begin{array}{c}
	-P_{1212}=P_{1221}=2P_{1313}=-2P_{1331}=-\frac{2r_0^2sin^2 r}{3L_0}, \\ P_{1414}=\frac{r_0^2sin^2 r sin^2\theta}{3}, \; P_{1441}=-\frac{r_0^2sin^2 r sin^2\theta}{3L_0}, \; P_{2323}=-\frac{r_0^2}{3},\\ P_{1331}=\frac{r_0^2}{3L_0}, \; P_{2442}=\frac{r_0^2 sin^2 \theta}{3L_0},\\ P_{2424}=-\frac{1}{2}P_{3434}=\frac{1}{2}P_{3443}=-\frac{r_0^2 sin^2\theta}{3}.
\end{array}\right.
\ee
If we study the Ricci tensor of the metric \eqref{CNM} and the operation of the Ricci operator on $R$, $C$ and $P$, we get the following:
\begin{pr}\label{Ricci}
The Ricci tensor of $CNS$ admits the following geometric properties:\\
(i) rank of $(S_{ab}-\alpha g_{ab})$ is $2$ \;for\; $\alpha=-\frac{1}{r_0^2}$, \\
(ii) $S_{ab}=\alpha g_{ab}+ \beta \Pi_{a}\otimes \Pi_{b}+ \gamma(\Pi_{a}\otimes \Phi_{b}+\Phi_{a} \otimes \Pi_{b})$ \;for\; $\alpha=-\frac{1}{r_0^2}$, $\beta=-1$, $\gamma=1$, $\Pi_a=\left\{\; -sin\,\theta, \; 1, \; 0 , \; 0 \; \right\}$ and $\Phi_a=\left\{ \; \frac{(1-2L_0)}{2L_0}sin\,\theta, \; \frac{1}{2L_0}, \; 0, \; 0 \; \right\}$, \; \;
\bea
\hspace{-6.5cm}(iii) \;\; \sum_{p,q,r}S^t_{p}R_{qrst}=0, \; \;  \sum_{p,q,r}S^t_{p}C_{qrst}=0 \; \; \sum_{p,q,r}S^t_{p}P_{qrst}=0.\nonumber
\eea
\end{pr}
From \eqref{g}, \eqref{rs} and \eqref{P}, we can get the non-vanishing local expressions (upto symmetry) of the tensors $P\cdot R$ and $Q(S,R)$ as follows:
\be
\begin{array}{c}
	(P\cdot R)_{121332}=-(P\cdot R)_{122331}=-\frac{r_0^2sin^2r}{3L_0}=-(P\cdot R)_{121323}=(P\cdot R)_{122313},\\ (P\cdot R)_{133441}=-\frac{1}{3}r_0^2sin^2rsin^2\theta =(P\cdot R)_{121442}=-L_0(P\cdot R)_{122441}\\ =-(P\cdot R)_{143431}=-L_0(P\cdot R)_{121424}=-(P\cdot R)_{133414}=L_0(P\cdot R)_{122414}=(P\cdot R)_{143413},\\ (P\cdot R)_{233424}=-\frac{1}{3}r_0^2sin^2\theta =(P\cdot R)_{243432}=-(P\cdot R)_{243423}=-(P\cdot R)_{233442}; \\ Q(S,R)_{122313}=\frac{r_0^2sin^2r}{L_0}=-L_0Q(S,R)_{121323},\; \; Q(S,R)_{143413}=r_0^2sin^2rsin^2\theta =L_0Q(S,R)_{122414}\\=-Q(S,R)_{133414} =Q(S,R)_{121424}, \; \; Q(S,R)_{243423}=-r_0^2sin^2\theta =-Q(S,R)_{233424}.
\end{array}\nonumber
\ee
From above we see that $CNS$ satisfies the condition $P\cdot R=-\frac{1}{3}Q(S,R)$.
\begin{thm}\label{CNMthm}
The CNS metric $\eqref{CNM}$ admits the following curvature restricted geometric properties: \\
(i) locally symmetric,\\
(ii) 2-quasi Einstein manifold,\\
(iii) generalized quasi-Einstein by Chaki \cite{C01}, \\
(iv) Ein(2)-space,\\
(v) Roter type manifold,\\
(vi) satisfies pseudosymmetric type condition $C\cdot R=\frac{(1+L_0)}{6r_0^2}Q(g,R)$ and $P\cdot R=-\frac{1}{3}Q(S,R)$ i.e., pseudosymmetric due to Weyl conformal curvature tensor and Ricci generalized projectively pseudosymmetric respectively,\\
(vii) Ricci tensor is Reimann compatible, conformal compatible, projectively compatible,\\
(vii)  the general form of compatible tensors for R, C, P, W, and K are given by
$$\left(\begin{array}{c}
\sigma_{11} \; \; \; \; \sigma_{12} \;\;\;\; 0  \; \; \; \;  0 \\ \sigma_{12} \;\;\;\; \sigma_{22} \;\;\;\; 0 \;\;\;\; 0 \\ 0 \;\;\;\;\;\; 0 \;\;\;\; \sigma_{33} \;\;\;\; \sigma_{34} \\ 0 \;\;\;\;\;\; 0 \;\;\;\; \sigma_{34} \;\;\;\; \sigma_{44}
\end{array}\right)$$
\mbox{where} $\sigma_{ij}$ \mbox{are arbitrary scalars}.
\end{thm}
\indent The energy momentum tensor $T$ of the $CNS$ metric \eqref{CNM} can be expressed in view of the Einstein field equations as
$$T=\frac{c^4}{8\pi G}[S-(\frac{\kappa}{2}-\Lambda)g]$$ where $ c$ denotes the speed of light in vacuum, and $G$ and $\Lambda$ are, respectively, the gravitational constant and the cosmological constant.\\
The non-zero components $T_{ab}$ of the tensor $T$ are given by
$$T_{11}=-\frac{c^4}{16\pi G L_0}(3+L_0)sin^2 r, \hspace{1.5cm} T_{22}=\frac{c^4}{16\pi G L_0}(3+L_0),$$ $$T_{33}=\frac{c^4}{16\pi G}(1+3L_0),\hspace{1.5cm} T_{44}=\frac{c^4}{16\pi G}(1+3L_0)sin^2 \theta .$$
Since $T$ is a linear combination of $S$ and $g$ and $\nabla S=0$, we have $\nabla T=0$ and hence we can state the following:
\begin{thm}
The  energy momentum tensor T of the charged Nariai metric \eqref{CNM} fulfills the following:\\
(i) covariant derivative of T vanishes, \\
(ii) $T=\alpha e_1 \otimes e_1+ \beta e_2\otimes e_2+ \gamma e_3\otimes e_3+ \mu e_4\otimes e_4$ \; for \; $\alpha=-1$, \; $\beta= \frac{1}{L_0}$, \; $\gamma=\frac{1}{4}$,\; $\mu=1$, $e_1=\left\{\frac{\sqrt{3+L_0}}{4\sqrt{L_0}}sin \,r,0,0,0 \right\}$, $e_2=\left\{0, \frac{\sqrt{3+L_0}}{4},0 ,0 \right\}$, $e_3=\left\{ 0, 0, \sqrt{1+3L_0}, 0 \right\}$ \; and \; $e_4=\left\{ 0, 0, 0, \frac{\sqrt{1+3L_0}}{4}sin\, \theta \right\}$ with $\|e_1\|=-\frac{(3+L_0)}{16r_0^2}$, $\|e_2\|=\frac{L_0(1+3L_0)}{16 r_0^2}$, $\|e_3\|=\frac{(1+3L_0)}{r_0^2}$ and $\|e_4\|=\frac{(1+3L_0)}{16r_0^2}$,\\
(iii) the tensor $T$ is $R$-compatible, $C$-compatible and $P$-compatible.
\end{thm}
The $NS$ metric \eqref{NNM} is a special case of the $CNS$ metric \eqref{CNM}. Thus for the metric \eqref{NNM} we have the following:
\begin{cor}\label{NNMcr}
The $NS$ metric \eqref{NNM} satisfies\\
(i) $S=-\frac{1}{r_0^2}$ g i.e., Einstein manifold and hence C = P = W,\\
(ii) $\nabla R=0$ and hence $\nabla C=\nabla W=0$ and \\
(iii) $C\cdot R=\frac{1}{3r_0^2}Q(g,R).$
\end{cor}
\begin{rem}\label{CNMrem1}
The metric \eqref{CNM} does not fulfill the following structures:\\
(a) $R$--space or $C$--space or $P$--space or $W$--space or $K$--space by Venzi,\\
(b) $C\cdot Z \ne 0$ for $Z=R$, $C$, $P$, $W$, $K$, $S$,\\
(c) $P\cdot Z \ne 0$ for $Z=R$, $C$, $P$, $W$, $K$,\\
(d) $W\cdot Z \ne 0$ for $Z=R$, $C$, $P$, $W$, $K$, $S$,\\
(e) $K\cdot Z \ne 0$ for $Z=R$, $C$, $P$, $W$, $K$, $S$,\\
(f) Einstein or quasi-Einstein.
\end{rem}
\begin{rem}\label{CNMrem2}
It is interesting to note that the metric \eqref{CNM} is conformally (resp., projectively, concircularly and conharmonicaly) symmetric but does not satisfy the semisymmetric type condition due to conformal (resp., projective, concircular and conharmonic) curvature tensor.
\end{rem}
\begin{rem}\label{CNMrem3}
The charged anti-Nariai metric with respect to the coordinates $(\tau,\rho,\omega, \psi)$ is given by $$ds^2=\frac{a_0^2}{K_0}(-sinh^2\rho d\tau+ d\rho^2)+a_0^2(d\omega^2+ sinh^2\omega d\psi^2)$$ where $a_0 > 0$ and $1 \le K_0 < 2$ are two constants. The above metric reduces to the form of $CNS$ mertric \eqref{CNM} by the coordinate transformation $\tau=t$, $\rho=ir$, $\omega=i\theta$, $\psi=\phi$ where $i=\sqrt{-1}$. Hence both the metrics possess the same curvature restricted geometric structures.
\end{rem}
%
%%%%%%%%%%%%%%%%%%%%%%%%%%%%%%%%%%%%%%%%%%%%%%%%%%%%%%%%%%%%%%%%%%%%%%%%%%%%%%%%%%%%%%%%%%%%%%%%%%%%%%%%%%%%%%%%%%%%%%%%%%%%
%                                          Curvature Restricted Geometric Structure                                        %
%%%%%%%%%%%%%%%%%%%%%%%%%%%%%%%%%%%%%%%%%%%%%%%%%%%%%%%%%%%%%%%%%%%%%%%%%%%%%%%%%%%%%%%%%%%%%%%%%%%%%%%%%%%%%%%%%%%%%%%%%%%%
%
\section{\bf{Geometric properties of CNS type metrics }}
%================================================================================%
%================================================================================%
Many authors have considered the (charged) Nariai spacetime metric in different forms suitable for their study. For example, to study the asymptotic behaviour of the Nariai spacetime the author of \cite{B09I} and \cite{B09II} have considered a family of Nariai type solutions and called them generalized Nariai solutions. Batista, in \cite{B16}, also studied the generalized form of charged Nariai solutions in arbitrary even dimensions. Here we consider a charged Nariai type metric which is also a direct topological product of two $2$-dimensional manifolds given by
\bea\label{CNTM}
ds^2=\frac{r_0^2}{L_0}(-\xi^2(r)dt^2+ dr^2)+ r_0^2(d\theta^2+ h^2(\theta) d\phi^2)
\eea
where $r_0$ , $L_0$ are constants defined in \eqref{CNM} and $\xi(r)$, $h(\theta)$ are everywhere non-vanshing smooth functions. From \eqref{CNTM} it is obvious that for $\xi(r)=sin\,r$ and $h(\theta)=sin \,\theta$, the above metric reduces to the metric \eqref{CNM}.\\
By a straightforward calculation we can get the non-zero local values (upto symmetry) of $R_{\alpha \beta \gamma \delta}$, $S_{\alpha \beta}$, $\kappa$ and $C_{\alpha \beta \gamma \delta}$ as follows:
\be\label{RSKC}
\left\{\begin{array}{c}
R_{1212}=\frac{r_0^2}{L_0}\xi \xi', \; R_{3434}=-r_0^2hh'; \; S_{11}=-\xi \xi'', \; S_{22}=\frac{\xi''}{\xi}, \\ S_{33}=\frac{h''}{h}, \; S_{44}=hh''; \;
\kappa=\frac{2}{r_0^2\xi h}(L_0h\xi''-\xi h''); \\ -\frac{L_0}{2}h^2C_{1212}=h^2C_{1313}=C_{1414}=-\xi^2h^2C_{2323} \\=-\xi^2C_{2424}=\frac{1}{L_0}\xi^2C_{3434}=-\frac{r_0^2\xi h}{6L_0}(L_0h\xi''+\xi h''). 
\end{array}\right.
\ee
\begin{pr}
	The conformal curvature of CNS type metric \eqref{CNTM} vanishes if $(L_0h\xi''+\xi h'')=0$.
\end{pr}
\begin{proof}
	Each component of the conformal curvature tensor, computed in \eqref{RSKC}, contains the factor $(L_0h\xi''+\xi h'')$ and hence the result follows.
	\end{proof}
Thus we have $U_C=\left\lbrace \; x \in M \; : \; (C)_x \ne 0 \;\right\rbrace = \left\lbrace \; (t,r,\theta,\phi) \; : \; L_0h\xi''+\xi h'' \ne 0 \; \right\rbrace$. Also the non-vanishing components of $\nabla_{a} R_{\alpha \beta \gamma \delta}$ and $\nabla_{a} S_{\alpha \beta}$ are given by
\be\label{crs}
\left\{
\begin{array}{c}
\nabla_1R_{1212}=-\frac{r_0^2}{L_0}(\xi'\xi''-\xi \xi'''), \;  \nabla_{1}R_{3434}=r_0^2(h' h''-h h'''); \nabla_{2}S_{11}=(\xi'\xi''- \xi \xi'''), \\ \nabla_{2}S_{22}=-\frac{1}{\xi^2}(\xi' \xi''-\xi \xi'''),\;  \nabla_{3}S_{33}=-\frac{1}{h^2}(h'h''-h h'''), \; \nabla_{4}S_{44}=-(h ' h''-h h''').
	\end{array}\right.
	\ee
	 where $\xi'=\frac{d\xi}{dr}$, \; $\xi''=\frac{d^2\xi}{dr^2}$, \; $\xi'''=\frac{d^3\xi}{dr^3}$, \; $h'=\frac{dh}{d\theta}$, \; $h''=\frac{d^2h}{d\theta^2}$ and $h'''=\frac{d^3h}{d\theta^3}.$\\
\begin{pr}
	For the CNS type metric \eqref{CNTM}, $\nabla R=0$ if $d\left( \frac{\xi}{\xi ''}\right)=0=d\left( \frac{h}{h''}\right)$ where $d$ is the differential operator.
\end{pr}
\begin{proof}
From \eqref{crs} we can write $\nabla_1R_{1212}=-\frac{r_0^2}{L_0}(\xi'\xi''-\xi\xi''')=-\frac{r_0^2}{L_0}\xi''^2d\left(\frac{\xi}{\xi''} \right) $ and similarly $\nabla_1 R_{3434}=r_0^2h''^2d\left( \frac{h}{h''}\right) $. Thus $\nabla R=0$ if $d\left(\frac{\xi}{\xi''}\right)=0=d\left( \frac{h}{h''}\right)$.
\end{proof}
We see that $\xi(r)=sin\,r$ and $h(\theta)=sin \,\theta$ satisfy $d\left(\frac{\xi}{\xi''}\right)=0=d\left( \frac{h}{h''}\right)$. Also from \eqref{crs} we note that the Ricci tensor admits the following: 
\begin{pr}\label{ricci}
(i) $\nabla_{a}S_{\alpha \beta}=\Pi_aS_{\alpha \beta} + \Phi_a g_{\alpha \beta}$ for $\Pi_a=\left\lbrace \; 0, \; -\frac{L_0h(\xi'\xi''-\xi\xi''')}{\xi(L_0h\xi''-\xi h'')}, \; \frac{\xi(h'h''-hh''')}{h(L_0h\xi''-\xi h'')}, \; 0 \; \right\rbrace $,\\ $\Phi=\left\lbrace \; 0, \; \frac{L_0h''(\xi'\xi''-\xi\xi''')}{r_0^2\xi(L_0h\xi''-\xi h'')}, \; -\frac{L_0\xi''\xi(h'h''-hh''')}{r_0^2h(L_0h\xi''-\xi h'')}, 0 \; \right\rbrace $ i.e., generalized Ricci-recurrent,\\
(ii) rank of $(S_{ab}-\alpha \; g_{ab})$ is $2$ for $\alpha=\frac{h''}{r_0^2h}$,\\
(iii) generalized quasi-Einstein in the sense of Chaki \cite{C01} for $\alpha=\frac{h^2}{r_0^2h}$, $\beta =-1$, $\gamma =1$, $\eta = \left\{\xi, 1, 0, 0\right\}$, $\delta =\left\{-\frac{L_0h\xi''-\xi h''-L_0\xi h}{2L_0h}, -\frac{L_0h\xi''-\xi h''+L_0\xi h}{2L_0\xi h}, 0, 0\right\}$ and
\bea
\hspace{-9.6cm}(iv) \;\; \sum_{p,q,r}S^t_{p}R_{qrst}=0, \; \; \sum_{p,q,r}S^t_{p}P_{qrst}=0.\nonumber
\eea
\end{pr}
%
%=================   gg, gS and SS     ========================================%
The tensors $g\wedge g$, $g\wedge S$ and $S\wedge S$ with respect to their local non-zero components (upto symmetry) are computed as below:
\be
\begin{array}{c}
L_0(g\wedge g)_{1212}=(g\wedge g)_{1313}=\frac{2r_0^4\xi^2}{L_0}, \; (g\wedge g)_{1414}=\frac{2r_0^4\xi^2h^2}{L_0}, \\ (g\wedge g)_{2323}=-\frac{2r_0^4}{L_0}, \; L_0(g\wedge g)_{2424}=(g\wedge g)_{3434}=-2r_0^4h^2 ;\\
(g\wedge S)_{1212}=\frac{2r_0^2\xi \xi''}{L_0}, \; (g\wedge S)_{3434}=-2r_0^2h h'',\\ h^2(g\wedge S)_{1313}=(g\wedge S)_{1414}=-\xi^2 h^2(g\wedge S)_{223}=-\xi^2(g\wedge S)_{2424}=\frac{r_0^2\xi h}{L_0}(L_0h \xi''+\xi h'') ;\\
(S\wedge S)_{1212}=2\xi''^2, \; (S\wedge S)_{3434}=-2h''^2 , \\ h^2(S\wedge S)_{1313}=(S\wedge S)_{1414}=-\xi^2h^2(S\wedge S)_{2323}=-\xi^2(S\wedge S)_{2424}=2\xi h \xi''h''.
\end{array}\nonumber
\ee
The Reimann tensor $R$ is linearly dependent with these tensors  and on $U_C$ it can be expressed as
$$R=\frac{r_0\xi h(L_0h \xi''+ \xi h'')}{2(L_0h \xi''-\xi h'')^2}\; (S\wedge S) -\frac{2L_0 \xi h\xi''h''}{(L_0h \xi''-\xi h'')^2}\; (g\wedge S) + \frac{L_0\xi''h''(L_0h\xi''+\xi h'')}{2r_0^2(L_0\xi''h-\xi h'')^2}\; (g\wedge g)$$
The above relation is the Roter type condition of $CNS$ type metric.
\begin{cor}\label{pseu}
	The metric \eqref{CNTM} is Roter type, hence from $Theorem$ $6.7$ of \cite{DGHS11} it satisfies the following geometric properties:\\
	(i) $R\cdot R=0$,\\
	(ii) $C\cdot R=\mathcal J'_R Q(g,R)$ for $\mathcal J'_R=-\frac{L_0h\xi''+\xi h''}{6r_0^2\xi h}$ and\\ (iii) $S^2 -\frac{L_0h\xi''+\xi h''}{r_0^2\xi h}\; S + \frac{L_0\xi''h''}{r_0^4\xi h}\; g=0$.
\end{cor}
\begin{rem}
As the CNS type metric \eqref{CNTM} is Roter type and generalized Ricci-recurrent, hence from $Theorem$ $3.8$ of \cite{SRK16}, it follows that the metric is super generalized recurrent (briefly, $SGK_4$) manifold. In fact it is special $SGK_4$.
	\end{rem}
 Also from $Theorem$ $3.9$ of \cite{SRK16}, the curvature tensors of the metric \eqref{CNTM} have the following recurrent properties:
 \begin{cor}\label{kn}
 (i) if $(L_0h\xi''-h''\xi)\ne 0$ then it is $SGK_4$ for $\Pi=\left\lbrace \; 0, \; -\frac{\omega_1}{2\xi\xi''}, \; -\frac{\omega_2}{2hh''}, \; 0 \; \right\rbrace $,\\ $\Phi=\left\lbrace \; 0, \; -\frac{r_0^2h\omega_1}{4\xi''\Omega}, \; \frac{r_0^2\xi \omega_2}{4h''\Omega}, \; 0 \; \right\rbrace $,\; $\Theta=\left\lbrace \; 0, \; \frac{L_0h''\omega_1}{4r_0^2\xi \Omega}, \; -\frac{L_0\xi''\omega_2}{4r_0^2h\Omega}, \; 0 \; \right\rbrace $.\\
 (ii) on $U_C$, the conformal tensor is recurrent for $\Pi=\left\{\; 0,\;  -\frac{L_0\omega_1 h}{\xi(\Omega+2h''\xi)}, \; -\frac{\xi \omega_2}{h(\Omega+2h''\xi)}, \; 0 \; \right\}$,\\
 (iii) if $(L_0h\xi''-h''\xi)\ne 0$ then it is $P$-$SGK_4$ for  $\Pi=\left\lbrace \; 0, \; -\frac{L_0h\omega_1}{\xi\Omega}, \; \frac{\xi\omega_2}{h\Omega}, \; 0 \; \right\rbrace ,\\ \Phi =\left\lbrace \; 0, \; \frac{L_0r_0^2\xi h^2h''\omega_1}{\Omega^3}, \; -\frac{r_0^2\xi^2\xi''h\omega_2}{\Omega^3}, \; 0\; \right\rbrace ,$ \; $\Psi=\left\lbrace \; 0, \; -\frac{L_0hh''\omega_1(\Omega+2h''\xi)}{\Omega^3}, \; \frac{L_0\xi\xi''\omega_2(\Omega+2h''\xi)}{\Omega^3}, \; 0\; \right\rbrace ,$\\ $\Theta=\left\lbrace \; 0, \; \frac{L_0h''\omega_1(\Omega^2-3L_0\xi\xi''hh'')}{3r_0^2\xi\Omega^3}, \; \frac{L_0\xi''\omega_2(\Omega^2-3L_0\xi\xi''hh'')}{3r_0^2h\Omega^3}, 0\; \right\rbrace ,$\\
 (iv) if $(L_0h\xi''-h''\xi)\ne 0$ then it is $W$-$SGK_4$ for $\Pi=\left\lbrace \; 0, \; -\frac{\omega_1}{2\xi h''}, \; \frac{\omega_2}{2hh''}, \; 0 \; \right\rbrace $, \; \\$\Phi=\left\lbrace \; 0,\; -\frac{r_0^2\omega_1}{4\xi''\Omega}, \; \frac{r_0^2\omega_2}{4h''\Omega}, \; 0 \;  \right\rbrace $, \; $ \Theta=\left\lbrace \; 0, \; \frac{\omega_1(\Omega^2 -6L_0\xi \xi''hh'')}{24r_0^2\xi^2\xi''h\Omega}, \; -\frac{\omega_2(\Omega^2 - 6L_0\xi \xi''hh'')}{24r_0^2\xi hh''\Omega}, \; 0 \; \right\rbrace $,\\
 (v) on $U_C$ its conharmonic tensor is recurrent for $\Pi=\left\lbrace \; 0, \; -\frac{L_0h\omega_1}{\xi(\Omega+2\xi h'')}, \; -\frac{\xi \omega_2}{h(\Omega+2\xi h'')}, \; 0 \;  \right\rbrace $\\ where $\omega_1=(\xi \xi''-\xi \xi''')$, $\omega_2=(h'h''-hh''')$ and $\Omega=(L_0h\xi''-h''\xi)$.
 \end{cor}
From the $Proposition$ \eqref{ricci}, $Corollary$ \eqref{pseu} and \eqref{kn}, we can state the following theorem:
\begin{thm}\label{CNTMthm}
The CNS type metric \eqref{CNTM} is (i) semisymmetric and  pseudosymmetric due to confomal curvature tensor (ii) 2-quasi Einstein, generalized quasi Einstein and Roter type manifold, (iii) special $SGK_4$ and $P$-$SGK_4$ manifold, (iv) $C$-$K_4$ and special $W$-$SGK_4$ manifold, (v) Reimann compatible as well as projective compatible.
\end{thm}
\begin{rem}\label{CNTMrem}
The metric \eqref{CNTM} does not satisfy the following geometric structures:\\
1. the metric is not R--space or C--space or P--space or W--space or K--space by Venzi,\\
2. it is not weakly symmetric and hence not Chaki pseudosymmetric,\\
3. it is neither recurrent nor generalized recurrent or hypergeneralized recurrent or weakly generalized recurrent,\\
4. its curvature 2-forms for Reimann curvature or projective curvature or concircular curvature is not recurrent,\\
5. $div R\ne 0$, $div C\ne 0$, $div P\ne 0$, $div W\ne 0$, $div K\ne 0$,\\
6. the Ricci tensor is neither conformally nor concircularly or cnharmonicaly compatible.
\end{rem}
%%%%%%%%%%%%%%%%%%%%%%%%%%%%%%%%%%%%%%%%%%%%%%%%%%%%%%%%%%%%%%%%%%%%%%%%%%%%%%%%%%%%%%%%%%%%%%%%%%%%%%%%%%%%%%%%%%%%%%%%%%%%%
%                                                     CONCLUSION                                                            %
%%%%%%%%%%%%%%%%%%%%%%%%%%%%%%%%%%%%%%%%%%%%%%%%%%%%%%%%%%%%%%%%%%%%%%%%%%%%%%%%%%%%%%%%%%%%%%%%%%%%%%%%%%%%%%%%%%%%%%%%%%%%%
%
\section{\bf{Conclusion}}
The charged Nariai spacetime is an exact solution of Einstein-Maxwell field equations and mathematically a topological product spacetime metric. In this paper, the curvature properties of the charged Nariai spacetime have been determined and it is seen that such spacetime is locally symmetric and satisfies the pseudosymmetric type conditions $C\cdot R=\frac{1+L_0}{6r_0^2}Q(g,R)$ and $P\cdot R=-\frac{1}{3}Q(S,R)$. Also it is a $4$-dimensional 2--quasi Einstein, generalized quasi-Einstein and Roter type manifold. The curvature properties of a charged Nariai type metric have also been investigated and it is found that such a metric is not locally symmetric but semisymmetric and its Ricci tensor is neither Codazzi nor cyclic parallel or recurrent but generalized recurrent. Also under certain restrictions, it admits recurrent type structures on several curvature tensors such as $SGK_4$, $W$-$SGK_4$, $P$-$SGK_4$, $C$-$K_4$ and $K$-$K_4$. This investigation proposes the existence of a class of semisymmetric and generalized recurrent type semi Reimannian manifolds and the charged Nariai type metrics represent that class.\\
\textbf{Acknowledgement.} The authors would like to express their gratitude to Deanship of Scientific Research at King Khalid University, Abha, Saudi Arabia for providing funding research groups under the research grant number R. G. P.$2/57/40$. The computations of various tensors and their covariant derivatives have been done by Wolfram Mathematica.
%%%%%%%%%%%%%%%%%%%%%%%%%%%
%
%%%%%%%%%%%%%%%%%%%%%%%%%%%%%%%%%%%%%%%%%%%%%%%%%%%%%%%%%%%%%%%%%%%%%%%%%%%%%%%%%%%%%%%%%%%%%%%%%%%%%%%%%%%%%%%%%%%%%%%%%%%%%%%%%%%%%%%%%%%%%%%%%%%
%
%%%%%%%%%%%%%%%%%%%%%%%%%%%%%%%%%%%%%%%%%%%%%%%%%%%%%%%%%%%%%%%%%%%%%%%%%%%%%%%%%%%%%%%%%%%%%%%%%%%%%%%%%%%%%%%%%%%%%%%%%%%%%%%%%%%%%%%%%%%%%%%%%%%

%%%%%%%%%%%%%%%%%%%%%

%%%%%%%%%%%%%%%%%%%%%%%%%%%%%%%%%%%%%%%%%%%%%%%%%%%%%%%%%%%%%%%%%%%%%%%%%%%%%%%%%%%%%%%%%%%%%%%%%%%%
\end{document}